\documentclass[11pt]{amsart}
\usepackage[english]{babel}
\usepackage{amsfonts,latexsym,amsthm,amssymb,graphicx}
\usepackage{mathabx}
\usepackage[all]{xy}
\usepackage[usenames]{color}
\usepackage{tikz}
\usetikzlibrary{shapes}
\usetikzlibrary{decorations.pathreplacing}
\usepackage{amsmath}
\usepackage{graphicx}
\usepackage[enableskew]{youngtab}
\usepackage[utf8]{inputenc}
\usepackage[english]{babel}
\usepackage{tikz}
\usetikzlibrary{arrows}
\usepackage{float}
\usepackage{amscd}
\usepackage{pstricks}
\usepackage{tableau}
\usepackage{multicol}
\usepackage{geometry}
\usepackage[alphabetic]{amsrefs}
\usepackage{enumerate}
\usepackage{pdfsync}
\usepackage[colorlinks=true, linkcolor=blue, citecolor=blue, urlcolor=blue]{hyperref}
\usepackage{dynkin-diagrams}
\usetikzlibrary{decorations.pathreplacing}

\usepackage{xcolor}

\newcommand{\IG}{\mathrm{IG}}

\newcommand{\QH}{\mathrm{QH}}

\DeclareFontFamily{OT1}{rsfs}{}
\DeclareFontShape{OT1}{rsfs}{n}{it}{<-> rsfs10}{}
\DeclareMathAlphabet{\mathscr}{OT1}{rsfs}{n}{it}

\CompileMatrices

\newtheorem{thm}{Theorem}[section]
\newtheorem{lemma}[thm]{Lemma}

\newtheorem{prop}[thm]{Proposition}

\newtheorem{conj}{Conjecture}

{\theoremstyle{definition} \newtheorem{defn}[thm]{Definition}}
{\theoremstyle{remark} \newtheorem{remark}[thm]{Remark}
}

\begin{document}

\title[]{Positivity determines the quantum cohomology of the odd symplectic Grassmannian of lines}



\author{Ryan M. Shifler}

\address{
Department of Mathematical Sciences,
Henson Science Hall, 
Salisbury University,
Salisbury, MD 21801
}
\email{rmshifler@salisbury.edu}


\begin{abstract}
Let $\IG:=\IG(2,2n+1)$ denote the odd symplectic Grassmannian of lines which is a horospherical variety of Picard rank 1. The quantum cohomology ring $\QH^*(\IG)$ has negative structure constants. For $n\geq 3$, we give a positivity condition that implies the quantum cohomology ring $\QH^*(\IG)$ is the only quantum deformation of the cohomology ring $\mbox{H}^*(\IG)$ up to the scaling of the quantum parameter. This is a modification of a conjecture by Fulton.
\end{abstract}

\maketitle


%
%

\section{Introduction}

Let $\IG:=\IG(2,2n+1)$ denote the odd symplectic Grassmannian of lines which is a horospherical variety of Picard rank 1. This is the parameterization of two dimensional subspaces of $\mathbb{C}^{2n+1}$ that are isotropic with respect to a general skew-symmetric form. The quantum cohomology ring $(\QH^*(\IG),\star)$ is a graded algebra over $\mathbb{Z}[q]$ where $q$ is the quantum parameter and $\deg q = 2n$. The ring has a Schubert basis given by $\{\tau_\lambda: \lambda \in \Lambda \}$ where \[\Lambda:=\{(\lambda_1,\lambda_2): 2n-1 \geq \lambda_1 \geq \lambda_2 \geq -1,\mbox{ }  \lambda_1>n-2 \Rightarrow \lambda_1>\lambda_2, \mbox{ and } \lambda_2=-1 \Rightarrow \lambda_1=2n-1 \}.\] We will often write $\tau_i$ in place of $\tau_{(i,0)}$. We define $|\lambda|=\lambda_1+\lambda_2$ for any $\lambda \in \Lambda$. Then $\deg(\tau_\lambda)=|\lambda|$. The ring multiplication is given by $\tau_\lambda \star \tau_\mu=\sum_{\nu,d} c_{\lambda,\mu}^{\nu,d}q^d \tau_\nu$ where $c_{\lambda,\mu}^{\nu,d}$ is the degree $d$ Gromov-Witten invariant of $\tau_\lambda$, $\tau_\mu$, and the Poicar\'e dual of $\tau_\nu$. Unlike the homogeneous $G/P$ case, the Gromov-Witten invariants may be negative. For example, in $\IG(2,5)$, we have \[\tau_{(3,-1)}\star\tau_{(3,-1)}=\tau_{(3,1)}-q \mbox{ and } \tau_{(2,1)}\star\tau_{(3,-1)}=-\tau_{(3,2)}+q\tau_1.\] The quantum Pieri rule has only non-negative coefficients and is stated in Proposition \ref{pieri:prop}. See \cite{pech:quantum, mihalcea.shifler:qhodd, GPPS} for more details on $\IG$.

\begin{defn}
For any given collection of constants $\{a_\mu \in \mathbb{Q}: \mu \in \Lambda \}$, a quantum deformation with the corresponding basis $\{ \sigma_\lambda : \lambda \in \Lambda \}$ is defined as a solution to the following system:
\[ \displaystyle \tau_\lambda=\sigma_\lambda+ \sum_{j\geq 1} \left( \sum_{|\mu|+2nj=|\lambda|} a_\mu q^j\sigma_\mu \right), \lambda \in \Lambda.\]
\end{defn}
\begin{remark}
It is always possible to re-scale the quantum parameter $q$ by a positive factor $\alpha>0$, or equivalently, multiply each Gromov-Witten invariant $c_{\lambda,\mu}^{\nu,d}$ by $\alpha^{-d}$. We only consider the $\alpha=1$ case in this manuscript.
\end{remark}

To contextualize the significance of quantum deformations we review the following conjecture by Fulton for Grassmannians and its extension to a more general case by Buch and Wang in \cite[Conjecture 1]{buch.Pos.Grass}.
\begin{conj}
Let $X = G/P$ be any flag variety of simply laced Lie type. Then the Schubert basis of $\QH^*(X)$ is the only homogeneous $\mathbb{Q}[q]$-basis that deforms the Schubert basis of $\mbox{H}^*(X,\mathbb{Q})$ and multiplies with non-negative structure constants.
\end{conj}

This conjecture is shown to hold for any Grassmannian and a few other examples in \cite{buch.Pos.Grass}. Li and Li proved the result for symplectic Grassmannians $\IG(2,2n)$ with $n\geq 3$ in \cite{PosIG}. The condition that the root system of $G$ be simply laced is necessary since the conjecture fails to hold for the Lagrangian Grassmannian $\IG(2,4)$ as shown in \cite[Example 6]{buch.Pos.Grass}. However, this conjecture is not applicable to $\IG(2,2n+1)$ since negative coefficients appear in quantum products for any $n$. We are able to modify the positivity condition on Fulton's conjecture to arrive at a uniqueness result for quantum deformations.

\begin{defn}
For $\IG(2,2n+1)$ we will use (**) to denote the condition that the coefficients of the quantum multiplication of $\sigma_{(1,1)}$ and any $\sigma_\mu$ in the basis $\{\sigma_\lambda: \lambda \in \Lambda \}$ are polynomials in $q$ with non-negative coefficients.
\end{defn}
We are ready to state the main result.
\begin{thm} \label{main:thm}
Let $n\geq 3$. Suppose that $\{\sigma_\lambda: \lambda \in \Lambda \}$ is a quantum deformation of the Schubert basis $\{\tau_\lambda: \lambda \in \Lambda \}$ of $\QH^*(\IG)$ such that Condition (**) holds. Then $\tau_\lambda=\sigma_\lambda$ for all $\lambda \in \Lambda$.
\end{thm}

\begin{remark}
The methods used in this manuscript are motivated by those of Li and Li in \cite{PosIG}. In particular, multiplication by $\tau_{(1,1)}$, which is not a generator, is all that we use to establish the uniqueness of the quantum deformation.
\end{remark}

In Section \ref{0:sec} we prove the main result for the $|\lambda|<2n$ case, state the quantum Pieri rule, and give identities for later in the paper; in Section \ref{2:sec} we prove the main result for the $|\lambda|=2n$ case; and in Section \ref{3:sec} we prove the main result for the $|\lambda|>2n$ case. Theorem \ref{main:thm} follows from Propositions \ref{1:prop}, \ref{2:prop}, and \ref{4:prop}.

{\em Acknowledgements.} I would like to thank an anonymous referee for identifying a gap in the argument for the $|\lambda|>2n$ case. I would also like to thank Leonardo Mihalcea for a very useful conversation. 

\section{Preliminaries} \label{0:sec}

We begin the section with a proposition that reduces the number of possible quantum deformations that we need to check. This is accomplished by using the grading. The proposition also states our main result for the $|\lambda|<2n$ case.
\begin{prop}\label{1:prop} We have the following results.
\begin{enumerate}
\item  We have that $\tau_\lambda=\sigma_\lambda+ \sum_{|\mu|+2n=|\lambda|} a_{\mu}q\sigma_{\mu}.$
\item If $|\lambda|<2n$ then $\tau_\lambda=\sigma_\lambda$.
\end{enumerate}
\end{prop}

\begin{proof} The first part follows since $|\lambda| \leq \dim (\IG(2,2n+1)=4n-3<4n=2\deg q$ for any $\lambda \in \Lambda$. The second part follows immediately from the grading.
\end{proof}

Next we state the quantum Pieri rule for $\IG(2,2n+1)$.
\begin{prop}\cite[Theorem 1]{pech:quantum} \label{pieri:prop}
The quantum Pieri rule.
\begin{equation*}
\tau_1 \star \tau_{a,b}= \left\{
        \begin{array}{ll}
            \tau_{a+1,b}+\tau_{a,b+1} & \quad \mbox{if }a+b \neq 2n-3 \mbox{ and }a \neq 2n-1,\\
            \tau_{a,b+1}+2\tau_{a+1,b}+\tau_{a+2,b-1} & \quad \mbox{if }a+b=2n-3,\\
            \tau_{2n-1,b+1}+q\tau_b& \quad \mbox{if } a=2n-1 \mbox{ and } 0 \leq b \leq 2n-3,\\
            \tau_{2n-1} & \quad a=2n-1 \mbox{ and } b=-1,\\
            q(\tau_{2n-1,-1}+\tau_{2n-2})& \quad a=2n-1 \mbox{ and } b=2n-2.
            
        \end{array}
    \right.
\end{equation*}
\begin{equation*}
\tau_{1,1} \star \tau_{a,b}= \left\{
        \begin{array}{ll}
            \tau_{a+1,b+1} & \quad \mbox{if }a+b \neq 2n-4,2n-3 \mbox{ and }a \neq 2n-1,\\
            \tau_{a+2,b}+\tau_{a+1,b+1}& \quad \mbox{if }a+b=2n-4 \mbox{ or }2n-3,\\
            q\tau_{b+1}& \quad \mbox{if } a=2n-1 \mbox{ and } b \neq 2n-3,\\
            q(\tau_{2n-1,-1}+\tau_{2n-2})& \quad a=2n-1 \mbox{ and } b=2n-3.
        \end{array}
    \right.
\end{equation*}

\end{prop}

\begin{lemma} \label{ident:lem}We have the following identities.
\begin{enumerate} 
\item Let $t\leq n-2$. Then \[\sigma_{(t,t)}=\tau_{(t,t)}=\Pi_{i=1}^t \tau_{(1,1)}=\Pi_{i=1}^t \sigma_{(1,1)}.\]
\item Let $|\lambda|\geq 2n$ and $t:=2n-\lambda_1$. 
\begin{enumerate}
\item If $\lambda_2+t \neq 2n-2$. Then
\[\left(\Pi_{i=1}^{t} \tau_{(1,1)} \right)\star\tau_\lambda=q\tau_{(\lambda_2+t)}.\]
\item If $\lambda_2+t = 2n-2$. Then
\[ \left(\Pi_{i=1}^{t} \tau_{(1,1)} \right)\star\tau_\lambda=q\tau_{(2n-1,-1)}+q\tau_{(2n-2)}.\]
\end{enumerate}
\item We have that
\[\displaystyle \Pi_{i=1}^{n-1} \tau_{(1,1)}=\tau_{(n,n-2)}.\]


\item If $2t+|\mu|\leq 2n-3$ and $t \leq n-2$ then
\[ \left(\Pi_{i=1}^t \tau_{(1,1)}\right) \star \tau_\mu=\tau_{(\mu_1+t,\mu_2+t)}.\]

\item If $2t+|\mu|=2n-2$ or $2n-1$ and $t \leq n-2$ then
\begin{eqnarray*}
\left(\Pi_{i=1}^t \tau_{(1,1)}\right) \star \tau_\mu=+\tau_{(\mu_1+t+1,\mu_2+t-1)}+\tau_{(\mu_1+t,\mu_2+t)}.
\end{eqnarray*}

\end{enumerate}
\end{lemma}

\begin{proof}
Part (1) is clear since $2t\leq 2n-4$. For Part (2), $\tau_{(1,1)}\star\tau_{(t-1,t-1)}\star \tau_{\lambda}=\tau_{(1,1)}\star \tau_{(2n-1,\lambda_2+t-1)}=q\tau_{(\lambda_2+t)}$ or $\tau_{(1,1)}\star\tau_{(t-1,t-1)}\star \tau_{\lambda}=\tau_{(1,1)}\star \tau_{(2n-1,\lambda_2+t-1)}=q\tau_{(2n-1,-1)}+q\tau_{(2n-2)}.$ For Part (3),
$\tau_{(1,1)}\star \Pi_{i=1}^{n-2} \tau_{(1,1)}=\tau_{(1,1)}\star\tau_{(n-2,n-2)}=\tau_{(n,n-2)}.$  Part (4) is clear. For Part (5), we have $\tau_{(1,1)}\star \left(\Pi_{i=1}^{t-1} \tau_{(1,1)}\right) \star \tau_\mu=\tau_{(1,1)}\star \tau_{(\mu_1+t-1,\mu_2+t-1)}=\tau_{(\mu_1+t,\mu_2+t)}+\tau_{(\mu_1+t+1,\mu_2+t-1)}.$ This completes the proof.
\end{proof}




\section{The $|\lambda|=2n$ case} \label{2:sec}
In this section we will assume that $|\lambda|=2n$. The main proposition of this section is stated next.

\begin{prop} \label{2:prop}
Let $|\lambda|=2n$. If $\tau_\lambda=\sigma_\lambda+aq$ and Condition (**) holds then $\tau_\lambda=\sigma_\lambda$.
\end{prop}

\begin{proof}
By Proposition \ref{1:prop} it must be the case that $\tau_{\lambda}=\sigma_\lambda+aq.$ We show $a\leq 0$ in two parts. Lemma \ref{a<=0:Lem} considers the $\lambda_1\geq n+2$ case and Lemma \ref{a<=02:lem} considers the $\lambda=(n+1,n-1)$ case. We show $a \geq 0$ in Lemma \ref{a>=0:Lem} as a straightforward application of the quantum Pieri rule. This completes the proof.
\end{proof}

\begin{lemma} \label{a<=0:Lem}
Let $|\lambda|=2n$ and $\lambda_1 \geq n+2$. If $\tau_\lambda=\sigma_\lambda+aq$ and Condition (**) holds then $a \leq 0$.
\end{lemma}

\begin{proof}
Let $t:=2n-\lambda_1 \leq n-2$. Note that $t+\lambda_1=2n$. Then we have the following by multiplying $\sigma_\lambda=\tau_\lambda-aq$ by $\left( \Pi_{i=1}^t \sigma_{(1,1)} \right)$ and using Part (1) of Lemma \ref{ident:lem}. \[ \left( \Pi_{i=1}^t \sigma_{(1,1)} \right) \star \sigma_\lambda=\tau_{(t,t)}\star\tau_\lambda-a\sigma_{(t,t)}q. \] By Part (2) of Lemma \ref{ident:lem} we have
$\tau_{(t,t)}\star\tau_\lambda=q\tau_{(\lambda_2+t)}=q\sigma_{(\lambda_2+t)}.$ 
So, \[ \left( \Pi_{i=1}^t \sigma_{(1,1)} \right) \star \sigma_\lambda=q\sigma_{(\lambda_2+t)}-a\sigma_{(t,t)}q. \]
It follows from Condition (**) that $a \leq 0$.
\end{proof}

We will now prove $a \leq 0$ for the $\lambda=(n+1,n-1)$ case.
\begin{lemma} \label{a<=02:lem}
Let $\lambda=(n+1,n-1)$. If $\tau_\lambda=\sigma_\lambda+aq$ and Condition (**) holds then $a \leq 0$.
\end{lemma}

\begin{proof}
Recall from Part (3) of Lemma \ref{ident:lem} that $\Pi_{i=1}^{n-1} \tau_{(1,1)}=\tau_{(n,n-2)}$ and from Part (2) of Lemma \ref{ident:lem} we have that
$\displaystyle \left(\Pi_{i=1}^{n-1} \tau_{(1,1)} \right)\star\tau_{\lambda}=q\tau_{(2n-1,-1)}+q\tau_{(2n-2)}.$ 
Multiplying $\sigma_\lambda=\tau_\lambda-aq$ by $\left(\Pi_{i=1}^{n-1} \sigma_{(1,1)}\right)$ and substituting in the identities yields
 \begin{eqnarray*}
 \left(\Pi_{i=1}^{n-1} \sigma_{(1,1)} \right)\star\sigma_{\lambda}&=&\left(\Pi_{i=1}^{n-1} \tau_{(1,1)} \right)\star \tau_\lambda-a\left(\Pi_{i=1}^{n-1} \tau_{(1,1)} \right)q\\
 &=& q\tau_{(2n-1,-1)}+q\tau_{(2n-2)}-aq\tau_{(n,n-2)}\\
 &=&q\sigma_{(2n-1,-1)}+q\sigma_{(2n-2)}-aq\sigma_{(n,n-2)}.
 \end{eqnarray*}
 It follows from Condition (**) that $a \leq 0$.
\end{proof}

We conclude the section by showing that $a\geq0$ in the next lemma.

\begin{lemma} \label{a>=0:Lem}
Let $|\lambda|=2n$. If $\tau_\lambda=\sigma_\lambda+aq$ and Condition (**) holds then $a \geq 0$.
\end{lemma}

\begin{proof}
Let $\lambda^j=(n+1+j,n-1-j)$ for all $j=0,1,2,...,n-2$. Assume that $\tau_{\lambda^j}=\sigma_{\lambda^j}+a_jq.$ Then for all $0 \leq j \leq n-2$ it follows from the quantum Pieri rule that $\tau_{(1,1)} \star \tau_{(n+j,n-2-j)}=\tau_{\lambda^j}.$ Since $\tau_{(n+j,n-2-j)}=\sigma_{(n+j,n-2-j)}$ by Part (2) of Lemma \ref{1:prop}, we have that
\begin{eqnarray*}
    \sigma_{(1,1)}\star\sigma_{(n+j,n-2-j)}=\tau_{(1,1)}\star\tau_{(n+j,n-2-j)}=\tau_{\lambda^j}=\sigma_{\lambda^j}+a_jq.
\end{eqnarray*}
It follows from Condition (**) that $a_j\geq 0$ for all $j=0, \cdots, n-2$.
\end{proof}

\section{The $|\lambda|>2n$ case} \label{3:sec}

In this section we will assume that $|\lambda|>2n$. Recall that by Proposition \ref{1:prop} it must be the case that $\tau_\lambda=\sigma_\lambda+\sum_{|\mu|+2n=|\lambda|} a_{\mu}q\sigma_{\mu}.$

\begin{lemma} \label{3:lem} 
Let $|\lambda|>2n$. If $ \tau_\lambda=\sigma_\lambda+\sum_{|\mu|+2n=|\lambda|}a_{\mu}q\sigma_{\mu}$ and Condition (**) holds then $a_\mu \leq 0$ or there is a $\mu'$ such that $a_\mu+a_{\mu'} \leq 0$.
\end{lemma}

\begin{proof}
If $|\lambda|>2n$ then $\lambda_1 \geq n+1$. Let $t:=2n-\lambda_1 \leq n-1$. Let $A(\lambda)=\sigma_{\lambda_2+t}$ if $\lambda_2+t \neq 2n-2$ and $A(\lambda)=\sigma_{(2n-1,-1)}+\sigma_{(2n-2)}$ if $\lambda_2+t = 2n-2$. We will multiply  $\sigma_\lambda=\tau_\lambda-\sum_{|\mu|+2n=|\lambda|}a_{\mu}q\sigma_{\mu}$ by $\left(\Pi_{i=1}^t \tau_{(1,1)}\right)$.
By Part (2) of Lemma \ref{ident:lem} we have that  $\left(\Pi_{i=1}^t \tau_{(1,1)}\right)\star \tau_{\lambda}=qA(\lambda).$ Since $\lambda_2+t<\lambda_1+t=2n$, we have that \[\left(\Pi_{i=1}^t \sigma_{(1,1)}\right)\star \sigma_{\lambda}=qA(\lambda)-\left(\Pi_{i=1}^t \sigma_{(1,1)}\right)\star \left( \sum_{|\mu|+2n=|\lambda|} a_{\mu}q\sigma_{\mu} \right).\] Next observe that $2t+|\mu|=2t+|\lambda|-2n=2n-\lambda_1+\lambda_2 \leq  2n-1.$ So, one of the following must occur:
\begin{itemize}
\item If $2t+|\mu| \leq 2n-3$ then by Part (4) of Lemma \ref{ident:lem} we have
\[
\left(\Pi_{i=1}^t \sigma_{(1,1)}\right)\star \sigma_\mu=\left(\Pi_{i=1}^t \tau_{(1,1)}\right) \star \tau_\mu=\tau_{(\mu_1+t,\mu_2+t)}=\sigma_{(\mu_1+t,\mu_2+t)}.\]
\item If $2t+|\mu|=2n-1$ or $2n-2$ then by Part (5) of Lemma \ref{ident:lem} we have
\begin{eqnarray*}
\left(\Pi_{i=1}^t \sigma_{(1,1)}\right)\star \sigma_\mu&=&\left(\Pi_{i=1}^t \tau_{(1,1)}\right) \star \tau_\mu=\tau_{(\mu_1+t,\mu_2+t)}+\tau_{(\mu_1+t+1,\mu_2+t-1)}\\
&=&\sigma_{(\mu_1+t+1,\mu_2+t-1)}+\sigma_{(\mu_1+t,\mu_2+t)}.
\end{eqnarray*}
\end{itemize}
Then $P:=\left(\Pi_{i=1}^t \sigma_{(1,1)}\right) \star \sigma_{\lambda}$ equals the following where terms are omitted when they do not satisfy the ring grading.
\begin{eqnarray*}
P&=&qA(\lambda)-\left( \sum_{\substack{|\mu|+2n=|\lambda| \\ 2t+|\mu| \leq 2n-3}} a_{\mu}q\sigma_{(\mu_1+t,\mu_2+t)} \right)-\left( \sum_{\substack{|\mu|+2n=|\lambda| \\ 2t+|\mu|=2n-1}} a_{\mu}q\left(\sigma_{(\mu_1+t+1,\mu_2+t-1)}+\sigma_{(\mu_1+t,\mu_2+t)}\right) \right)\\
 &-&\left( \sum_{\substack{|\mu|+2n=|\lambda| \\ 2t+|\mu|=2n-2}} a_{\mu}q\left(\sigma_{(\mu_1+t+1,\mu_2+t-1)}+\sigma_{(\mu_1+t,\mu_2+t)}\right) \right).
 \end{eqnarray*}
We have the following two equalities that will be used to precisely write the summations for the $2t+|\mu|=2n-1$ and $2t+|\mu|=2n-2$ cases.

\[ \left(\Pi_{i=1}^t \sigma_{(1,1)}\right)\star \sigma_{(2n-1-2t-i,i)}= \sigma_{(2n-t-i,t-1+i)}+
      \sigma_{(2n-1-t-i,t+i)} \mbox{ for } 0\leq i<n-1-t.
      \]

\[ \left(\Pi_{i=1}^t \sigma_{(1,1)}\right)\star \sigma_{(2n-2-2t-i,i)}=\begin{cases} 
      \sigma_{(2n-1-t-i,t-1+i)}+\sigma_{(2n-2-t-i,t+i)} &: 0\leq i<n-1-t \\
      \sigma_{(n,n-2)} &: i=n-1-t.
   \end{cases}
\]
To simplify notation we will let $a_i=a_{(2n-1-2t-i,i)}$ and $b_i=a_{(2n-2-2t-i,i)}$ for $0 \leq i \leq n-1-t$. Then we have the following identities.

\begin{eqnarray*}
 \sum_{\substack{|\mu|+2n=|\lambda| \\ 2t+|\mu|=2n-1}} a_{\mu}q\left(\sigma_{(\mu_1+t+1,\mu_2+t-1)}+\sigma_{(\mu_1+t,\mu_2+t)}\right) &=&\sum_{i=0}^{n-1-t}a_iq\left(\sigma_{(2n-t-i,t-1+i)}+\sigma_{(2n-1-t-i,t+i)}\right).\\
 \sum_{\substack{|\mu|+2n=|\lambda| \\ 2t+|\mu|=2n-2}} a_{\mu}q\left(\sigma_{(\mu_1+t+1,\mu_2+t-1)}+\sigma_{(\mu_1+t,\mu_2+t)}\right) &=&\left(\sum_{i=0}^{n-2-t}b_iq\left(\sigma_{(2n-1-t-i,t-1+i)}+\sigma_{(2n-2-t-i,t+i)}\right)\right)\\
 &+&b_{n-1-t}\sigma_{(n,n-2)}.
\end{eqnarray*}
It follows that
\begin{eqnarray*}
 P&=&qA(\lambda)-\left( \sum_{\substack{|\mu|+2n=|\lambda| \\ 2t+|\mu| \leq 2n-3}} a_{\mu}q\sigma_{(\mu_1+t,\mu_2+t)} \right)-\left(\sum_{i=0}^{n-1-t}a_iq\left(\sigma_{(2n-t-i,t-1+i)}+\sigma_{(2n-1-t-i,t+i)}\right) \right)\\
 &-&\left( \sum_{i=0}^{n-2-t}b_iq\left(\sigma_{(2n-1-t-i,t-1+i)}+\sigma_{(2n-2-t-i,t+i)}\right) \right)-b_{n-1-t}\sigma_{(n,n-2)}.
\end{eqnarray*}
Reorganizing the second two sums yields the following equation.
\begin{eqnarray*} 
P&=&qA(\lambda)-\left( \sum_{\substack{|\mu|+2n=|\lambda| \\ 2t+|\mu| \leq 2n-3}} a_{\mu}q\sigma_{(\mu_1+t,\mu_2+t)} \right)\\
&-&a_0q\sigma_{(2n-t,t-1)}-\left(\sum_{i=0}^{n-2-t}(a_i+a_{i+1}) q\sigma_{(2n-1-t-i,t+i)}\right)-a_{n-1-t}q\sigma_{(n+1,n-2)}\\
&-& b_0q\sigma_{(2n-1-t,t-1)}-\left(\sum_{i=0}^{n-2-t}(b_i+b_{i+1})q \sigma_{(2n-2-t-i,t+i)}\right).
\end{eqnarray*}

When $|\mu|+2n=|\lambda|$ and $2t+|\mu| \leq 2n-3$, we have that $a_\mu \leq 0$ by Condition (**). In the remaining cases, notice by Condition (**) that $a_i+a_{i+1} \leq 0$ or $b_i+b_{i+1} \leq 0$ for all $0 \leq i \leq n-2-t$. The result follows.\end{proof}

\begin{prop}\label{4:prop}
Let $|\lambda|>2n$. If $ \tau_\lambda=\sigma_\lambda+\sum_{|\mu|+2n=|\lambda|}a_{\mu}q\sigma_{\mu}$ and Condition (**) holds then $a_\mu=0$.
\end{prop}

\begin{proof}
We proceed by induction. Suppose $\tau_\lambda=\sigma_\lambda$ for all $|\lambda| \leq s$ where $s \geq 2n$. Consider $|\lambda|=s+1.$ Since $|\lambda| \geq 2n+1$ for $|\lambda|=s+1$, and by an application of the quantum Pieri rule, we have that $\tau_{(1,1)} \star \tau_{(\lambda_1-1,\lambda_2-1)}=\tau_\lambda.$ Observe that $\tau_{(\lambda_1-1,\lambda_2-1)}=\sigma_{(\lambda_1-1,\lambda_2-1)}$ by the inductive hypothesis. Then \[ \sigma_{(1,1)}\star \sigma_{(\lambda_1-1,\lambda_2-1)}=\tau_{(1,1)} \star \tau_{(\lambda_1-1,\lambda_2-1)}=\tau_\lambda=\sigma_\lambda+\sum_{|\mu|+2n=|\lambda|} a_{\mu}q \sigma_{\mu}.\] So, $a_\mu \geq 0$ by Condition (**). By Lemma \ref{3:lem}, either $a_\mu \leq 0$ or there is a $\mu'$ such that $a_\mu+a_{\mu'} \leq 0$. In either case, this implies $a_\mu=0$ since $a_\mu \geq 0$ and $a_{\mu'} \geq 0$. The result follows.
\end{proof}


\bibliographystyle{halpha}
\bibliography{bibliography}

\end{document}